\documentclass[12pt,twoside]{amsart}
\usepackage{latexsym,amsmath,amsopn,amssymb,amsthm,amsfonts}
\usepackage[T2A]{fontenc}
\usepackage[cp1251]{inputenc}
\usepackage{graphicx}

\newtheorem{theorem}{Theorem}

\newtheorem{proposition}{Proposition}

\newtheorem{remark}{Remark}

\newtheorem{definition}{Definition}

\newtheorem{corollary}{Corollary}

\newtheorem{lemma}{Lemma}

\newcommand{\Ad}{\operatorname{Ad}}

\newcommand{\ad}{\operatorname{ad}}

\newcommand{\Ric}{\operatorname{Ric}}

\newcommand{\CD}{\operatorname{CD}}

\newcommand{\ess}{\operatorname{ess}}

\newcommand{\End}{\operatorname{End}}

\newcommand{\Nul}{\operatorname{Nul}}

\begin{document}

\title[Curvatures of homogeneous sub-Riemannian manifolds]
{On curvatures of homogeneous\\ sub-Riemannian manifolds}
\author{V.\,N.\,Berestovskii}
\thanks{The publication was supported by the Ministry of Education and Science of the Russian Federation (the Project number 1.3087.2017/PCh)}
\address{V.N.~Berestovskii}
\address{Sobolev Institute of Mathematics SB RAS, \newline 4 Akad. Koptyug avenue, 630090, Novosibirsk, Russia}
\address{Novosibirsk State University, \newline 2 Pirogov str., 630090, Novosibirsk, Russia}
\email{vberestov@inbox.ru}
\maketitle
\maketitle {\small
\begin{quote}
\noindent{\sc Abstract.}
The author proved in the late 1980s that any homogeneous manifold with an intrinsic 
metric is isometric to some homogeneous quotient space of a connected Lie group 
by its compact subgroup with an invariant Finslerian or sub-Finslerian metric. In 
a case of trivial compact subgroup, invariant Riemannian or sub-Riemannian metrics are 
singled out from invariant Finslerian or sub-Finslerian metrics by their one-to-one 
correspondence with special one-parameter Gaussian convolutions semigroups of 
absolutely continuous probability measures. Any such semigroup is generated by a
second order hypoelliptic operator. In connection with this, the author discusses
briefly the operator definition of Ricci lower bounds for sub-Riemannian manifolds 
by Baudoin-Garofalo. Earlier, Agrachev 
defined a notion of curvature for sub-Riemannian manifolds. As an alternative,
the author discusses in some detail the old definitions of the curvature tensors 
for rigged metrized distributions on manifolds given by Schouten, Wagner, and Solov'ev. 
To calculate the Solov'ev sectional and Ricci curvatures for homogeneous sub-Riemannian 
manifolds, the author suggests to use in some cases special riggings of invariant 
completely nonholonomic distributions on manifolds. As a justification, we find a 
foliation on the cotangent bundle $T^{\ast}G$ over a Lie group $G$ whose leaves are 
tangent to invariant Hamiltonian vector fields for the Pontryagin-Hamilton function. This 
function was applied in the Pontryagin maximum principle for the time-optimal problem. The 
foliation is entirely described by the co-adjoint representation of the Lie group $G.$ 
Also we use the canonical symplectic form on $T^{\ast}G$ and its values for the 
above mentioned invariant Hamiltonian vector fields. In particular, the above rigging method 
is applicable to contact sub-Riemannian manifolds, sub-Riemannian Carnot groups, 
and homogeneous sub-Riemannian manifolds possessing a submetry onto a Riemannian 
manifold. At the end, some examples are presented. 
\end{quote}}
{\small
\begin{quote}
\noindent{\textit{Key words and phrases:}} co-adjoint representation, contact 
form, cotangent bundle, Hamiltonian vector field, homogeneous sub-Riemannian manifold, 
left-invariant sub-Riemannian metric, Lie algebra, Lie group, Pontryagin-Hamilton function,
submetry, sub-Riemannian curvature, symplectic form.
\end{quote}}
{\small
\begin{quote}
\noindent{\textit{2010 Mathematics Subject Classification.}} Primary: 53C17, 58B20.\\
Secondary: 53C21, 53D05, 53D10.
\end{quote}}

\section*{Introduction}

With the help of the results by Iwasawa-Gleason-Yamabe on the structure of connected
locally compact topological groups, the author proved in the late 1980s
that every locally compact homogeneous space with intrinsic metric is a projective
limit of a sequence of homogeneous manifolds with an intrinsic metric \cite{Ber88}, 
\cite{Ber89}. In turn, any homogeneous manifold with an intrinsic metric is
isometric to some homogeneous quotient space $G/H$ of a connected Lie group 
$G$ by its compact subgroup $H$ with $G$-invariant {\it Finslerian} or 
{\it sub-Finslerian metric} $d$ \cite{Ber88.1}, \cite{Ber89.1}. The metric
$d$ is defined by some $G$-invariant completely nonholonomic (vector) 
distribution $D$ on $G/H$ and norm $F$ on $D.$ In the Finsler case $D=T(G/H).$
The distance $d(x,y)$ between any points $x,y\in G/H$ is equal to the infimum
of lengths of piece-wise smooth paths tangent to $D$ and joining these points.
By definition, the length of any such path $\gamma=\gamma(t),$ $a\leq t\leq b$ is 
equal to the integral $\int_a^bF(\stackrel{\cdot}{\gamma}(t))dt.$   

If $F$ is equal to the square root of scalar square with respect to the inner product 
$\langle\cdot,\cdot\rangle,$ then $d$ is a {\it Riemannian} or 
{\it sub-Riemannian metric}. For any $G$-invariant 
(sub-)Finslerian (respectively, (sub-)Riemannian) metric $d$ 
on $G/H$ there exists a $G$-left-invariant and $H$-right-invariant  
(sub-)Finslerian (respectively, (sub-)Riemannian) metric $\delta$ on $G$ 
such that the canonical projection $p: (G,\delta)\rightarrow (G/H,d)$ is a 
{\it submetry} \cite{BG}. In the Riemannian case this submetry is a Riemannian
submersion. 

In \cite{Ber88.1} the case $H=\{e\}$ is considered; then $d$ is a left-invariant
Finslerian or sub-Finslerian (or more special Riemannian or sub-Riemannian) 
metric on the Lie group $G$. In this case the smallest Lie algebra, containing
the vector subspace $D(e)$ of the Lie algebra $\mathfrak{g}=(T_eG,[\cdot,\cdot])$ 
of the Lie group $G,$ coincides with $\mathfrak{g},$ and $D$ is a left-invariant
vector subbundle of the tangent bundle $TG.$ As a consequence of the left-invariance
of $D$ and norm $F$, it is enough to assign $D(e)$ and a value $F$ on $D(e).$ 

These results together with {\it the Pontryagin maximum principle} for the
corresponding left-invariant {\it time-optimal problems} \cite{PBGM}, 
\cite{Ber88} on Lie groups permit in many cases to find (locally) shortest arcs
of homogeneous intrinsic metrics on manifolds.

It is difficult to study general homogeneous sub-Finslerian
manifolds and there are a few works on them. One can mention
papers by Berestovskii \cite{Ber94} and G.A.~Noskov \cite{Nos}; there are
found {\it geodesics}, i.e. {\it locally shortest (curves)}, and 
{\it shortest arcs} of arbitrary left-invariant sub-Finslerian metrics on
three-dimensional Heisenberg group.

Recently, A.A.~Agrachev defined a curvature of sub-Riemannian manifolds \cite{ABR}.
For this he applied a thorough, natural, justified, and universal approach. 
However, in the general case, at least now, there is no available formula to 
calculate this curvature by the Agrachev method. It is possible to do this for 
contact sub-Riemannian manifolds \cite{AgrLee}, \cite{ABR1}.

On the other hand, more than thirty years ago, my former colleague 
at Omsk State University A.F.~Solov'ev defined and suggested how to 
calculate easily the sectional and Ricci curvatures of any rigged and
metrized distribution $(D,\langle\cdot,\cdot\rangle)$ in manifolds.

In order to apply the Solov'ev method to the case of homogeneous sub-Riemannian 
manifolds, it is necessary to solve only one (generally difficult) problem, 
namely, to find a justified  invariant rigging $D^{\perp}$ of $D$, i.e. a
complementary distribution in $G/H$. 

We show that it is possible to apply the Solov'ev method in the 
following cases: 

1) for any smooth contact sub-Riemannian manifold,

2) for any three-dimensional Lie group with left-invariant 
sub-Riemannian metric, 

3) when there is a submetry from $(G/H,d)$ onto a Riemannian manifold, 

4) for sub-Riemannian Carnot groups,  

5) for sub-Riemannian $(G,d)$ when there is a unique rigging $D^{\perp}$ of $D$ 
such that $D^{\perp}(e)$ is a Lie subalgebra $\mathfrak{k}$ of Lie algebra 
of the Lie group $G$ and $[\mathfrak{k},D(e)]\subset D(e).$

{\it It is possible to show that for any homogeneous sub-Riemannian manifold} $(M,d)$
{\it there is a connected Lie group} $G$ {\it with a left-invariant sub-Riemannian metric} 
$\delta$ {\it such that there is a submetry from} $(G,\delta)$ {\it onto} $(M,d)$ 
{\it and, moreover, the problem of calculation of Solov'ev curvatures for} $(M,d)$ {\it 
is fully reduced to the case} of $(G,\delta)$. 

To justify the cases of the sub-Riemannian Lie group $G$, we shall find a special 
foliation on the cotangent bundle $T^{\ast}G$. Its leaves are tangent to invariant 
Hamiltonian vector fields for the Pontryagin-Hamilton function, applied in 
the Pontryagin maximum principle for the time-optimal problem. This foliation is 
transversal to the fibres of the canonical projection from $T^{\ast}G$ onto $G$. 
This projection maps any leaf of the foliation onto all $G.$ If a leaf contains 
covectors $\xi_0$ and $\xi_1$ over $e$ and $g\in G$ respectively, then 
$\xi_0=\Ad^{\ast}g(\xi_1).$ These properties entirely characterize
the foliation. Also we use the canonical symplectic structure on $T^{\ast}G$ which 
is really tightly connected with another well-known canonical
symplectic structure on orbits of the co-adjoint representation $\Ad^{\ast}$ of 
the Lie group $G.$ Notice that these considerations do not depend on a choice of 
a left-invariant Riemannian or sub-Riemannian metric on $G.$  

In the last chapter of the paper we shall consider some examples. It is interesting that
every Hopf bundle presents a particular case of situation 3) above.

The author thinks that applications of the Solov'ev method to sub-Riemannian manifolds
deserve attention because there appeared different notions of curvatures for 
these manifolds. Therefore it is useful to compare these notions and single out
the ``correct and applicable'' one between them. 

In connection with this, it is appropriate to give the following extensive quotation from
the end of the Introduction to the paper \cite{BaudGar} by F.~Baudoin and N.~Garofalo: 
``For general metric measure spaces, a different notion of lower bounds on the 
Ricci tensor based on the theory of optimal'' (Kantorovich-Monge mass) ``transportation
has been recently proposed independently by Sturm \cite{Sturm1}, \cite{Sturm2} and
by Lott-Villani \cite{LV} (see also \cite{Olliv}). However, as pointed out by Juillet
\cite{Jul}, the remarkable theory developed in those papers does not appear to be
suited for sub-Riemannian manifolds. For instance, in that theory the flat Heisenberg
group $\mathbb{H}^1$ has curvature $-\infty$. \dots An analysis shows that, interestingly,
our notion of the Ricci tensor, coincides, up to a scaling factor, with'' one given in 
\cite{ABR}, \cite{AgrLee}, \cite{ABR1}.      

\section{Preliminaries}

The following statements are true \cite{BP}: 

1) A locally compact homogeneous space with an intrinsic metric $(M,d)$ 
is isometric to a homogeneous Riemannian manifold of sectional curvature $\leq K$
for a number $K\in \mathbb{R}$ if and only if $(M,d)$ has Alexandrov 
curvature $\leq K$; 

2) there exist infinite-dimensional locally compact homogeneous spaces with
an intrinsic metric of Alexandrov curvature $\geq K > -\infty$; 

3) a finite-dimensional locally compact homogeneous space with intrinsic
metric $(M,d)$ is isometric to a homogeneous Riemannian manifold of sectional
curvature $\geq K$ for some number $K$ if and only if $(M,d)$ has Alexandrov 
curvature $\geq K$; 

4) if a locally compact homogeneous space with intrinsic metric $(M,d)$ has
curvature $\geq K >0$, then $(M,d)$ is isometric to some homogeneous Riemannian 
manifold with sectional curvature $\geq K$.

On the other hand, the author does not know any {\it natural geometric characterization
of sub-Riemannian metrics in the class of homogeneous sub-Finslerian metrics}. Possibly, 
there is no such characterization. 

Contemporary methods of probability theory, functional analysis, and partial
differential equations permit to set a one-to-one correspondence between Riemannian
or sub-Riemannian metrics on any given Lie group $G$ and {\it symmetric} in the
sense of H.~Heyer \cite{He} and E.~Siebert \cite{Sieb82} {\it one-parameter convolution 
Gaussian semigroups of absolutely continuous} (with respect to a left-invariant Haar 
measure $\omega$ on the group $G$) {\it probability measures} $\mu_t=u_t\omega$ 
with densities $u_t,$ $t>0$ \cite{He}, \cite{Sieb82}. Moreover, the function 
$u:\mathbb{R}_+\times G\rightarrow \mathbb{R}_+,$ $u(t,g)=u_t(g),$  
is a smooth (i.e. infinitely differentiable) solution of a linear hypoelliptic 
parabolic homogeneous partial differential equation $(\partial/\partial t-L)u=0$ 
similar to the heat equation \cite{Sieb82}. Here $L=\sum_{i=1}^{m}X_i^2,$ 
where $X_1,\dots ,X_m$ are left-invariant vector fields on $G$ generating the Lie
algebra $\mathfrak{g}$ \cite{He}, \cite{Sieb82}. Therefore $L$ is a left-invariant
linear {\it hypoelliptic operator} in the sense of H\"ormander \cite{Hor1}. 
The operator $L$ naturally corresponds to the left-invariant (sub-)Riemannian metric 
$d$ on $G$ defined by a distribution $D$ with orthonormal basis $X_1,\dots ,X_m.$ 
Conversely, let a left-invariant (sub-)Riemannian metric $d$ on $G$ be defined
by a distribution $D$ with orthonormal basis $X_1,\dots ,X_m.$ Then the operator
$L=\sum_{i=1}^{m}X_i^2$ assigns a unique smooth solution 
$u=u(t,\cdot),$ $t>0,$ of the differential equation $(\partial/\partial t-L)u=0$ 
such that $\mu_t=u_t\omega$ is a Gaussian convolution semigroup of absolutely
continuous probability measures on $G.$ 

One can easily see that $L 1 =0,$ $L$ is a symmetric and non-positive operator 
relative to $\omega$, i.e. for any $f,g\in C_0^{\infty}(G),$ 
$$\int_G fLg d\omega=\int_G gLf d\omega,\quad \int_G fLf d\omega \leq 0.$$

\section{On the operator definitions of curvatures}

In the paper \cite{BaudGar}, F.~Baudoin and N.~Garofalo 
introduced a generalized curvature-dimension inequality. We shall apply 
(only) definitions of this work to the case we are interested in, namely,
the Lie group $G=G^n$ with a left-invariant sub-Riemannian metric $d$ and 
a hypoelliptic operator $L=\sum_{i=1}^m X_i^2,$ $2\leq m < n.$ 

They associate with such $L$ the following symmetric differential bilinear form 
\begin{equation}
\label{BaG}
\Gamma(f,g)=\frac{1}{2}\{L(fg)-fLg - gLf\}.
\end{equation}
The expression $\Gamma(f):=\Gamma(f,f)=\sum_{i=1}^k(X_if)^2$ is 
{\it le carr\'e du champ} \cite{BaudGar} since  
$$d(x,y)=\sup \{|f(x)-f(y)|f\in C^{\infty}(G),\quad 
\|\Gamma(f)\|_{\infty}\leq 1\},$$
where $\|f\|_{\infty}=\ess \sup_G|f|$ for a smooth function $f$ on $G$.

In addition, in \cite{BaudGar} is given some symmetric bilinear differential
form of the first order 
$\Gamma^Z: C^{\infty}(G)\times C^{\infty}(G)\rightarrow \mathbb{R},$
such that 
$$\Gamma^Z(fg,h)=f\Gamma^Z(g,h)+g\Gamma^Z(f,h),\quad 
\Gamma^Z(f):=\Gamma^Z(f,f)\geq 0,\quad\Gamma^Z(1)=0.$$ 
In the case of a Lie group, it is natural to define it in the following manner. 
Let us assume that there are chosen a left-invariant {\it rigging}  of distribution $D$, i.e. 
a left-invariant distribution $D^{\perp}$ on $G$, complementary to $D$ such that
$D\oplus D^{\perp}=TG,$ a left-invariant scalar product $(\cdot,\cdot)$ on $TG$ such that 
$(\cdot,\cdot)_D=\langle\cdot,\cdot\rangle$ and $(D,D^{\perp})=0,$ 
and also a left-invariant basis of vector fields $Z_{1},\dots, Z_{l}$ in $D^{\perp}$, 
orthonormal relative to $(\cdot,\cdot)_{D^{\perp}}$, so that $m+l=n$. Notice that the
English term ``rigging'' was suggested by V.V.~Wagner in \cite{Vag}. We define 
$\Gamma^Z(f,g):=\sum_{j=1}^l Z_jf\cdot Z_jg.$ Then in \cite{BaudGar}
are defined  second order differential forms
\begin{equation}
\label{g2}
\Gamma_2(f,g)=\frac{1}{2}\{L\Gamma(f,g)-\Gamma(f,Lg)-\Gamma(g,Lf)\},
\end{equation}
\begin{equation}
\label{gz2}
\Gamma^Z_2(f,g)=\frac{1}{2}\{L\Gamma^Z(f,g)-\Gamma^Z(f,Lg)-\Gamma^Z(g,Lf)\}.
\end{equation}

\begin{definition}
\label{BDD}
It is said that in $G$ is satisfied a {\it generalized curvature-dimension inequality}
$\CD(\rho_1,\rho_2,\kappa,r)$ relative to $L$ and $\Gamma^Z$, if there exist constants
$\rho_1\in \mathbb{R},$ $\rho_2>0,$ $\kappa\geq 0,$ and $0<r\leq \infty$ such that
the inequality 
\begin{equation}
\label{bde}
\Gamma_2(f)+\nu\Gamma^Z_2(f)\geq \frac{1}{r}(Lf)^2 + 
\left(\rho_1-\frac{\kappa}{\nu}\right)
\Gamma(f)+ \rho_2\Gamma^Z(f)
\end{equation}     
is satisfied for all $f\in C^{\infty}(G)$ and every $\nu>0$.
\end{definition}

It should be pointed out that if in Definition \ref{BDD} we take 
$r=n=\dim(M),$ $L=\Delta,$ $\Gamma^Z\equiv 0,$ $\rho_1=\rho,$ 
$\kappa=0$ for a smooth Riemannian manifold $M$, then we obtain
the inequality $\CD(\rho,n)$ of {\it Bakry-Emery}. Bakry showed 
(see quotations in \cite{BaudGar}) that 
$\Ric (M)\geq \rho \Leftrightarrow \CD(\rho,n).$ Precisely this equivalence
served as the motivation for the work \cite{BaudGar} by Baudoin-Garofalo. 
The parameter $\rho_1$ plays the main role in the inequality (\ref{bde}) since 
in geometric examples, considered in \cite{BaudGar}, it represents the lower 
bound for the sub-Riemannian generalization Ricci curvature.

The article \cite{BaudGar} is based on (\ref{bde}) and the general Hypothesis 1,2,3.
Hypothesis 1 is equivalent to completeness of the metric $d$ which is satisfied
in our case. Hypothesis 2 is the following commutation relation:
\begin{equation}
\label{hyp2}
\Gamma(f,\Gamma^Z(f))=\Gamma^Z(f,\Gamma(f))\quad\mbox{for all}
\quad f\in C^{\infty}(G).
\end{equation} 
Hypothesis 3 has a technical character. It is enough to say that it is valid
for $G$ in consequence of the work \cite{Sieb84} E.~Siebert. 

\section{On definitions of curvatures for rigged metrized distributions}

In the papers \cite{SchVan} by Schouten and van Kampen and \cite{Vag} 
by Wagner the authors introduced and studied curvature tensors of
nonholonomic manifolds. It is not easy to read and understand these articles because
of the coordinate presentation of notions and results. In the paper \cite{Gorb} by
E.M.~Gorbatenko there was given a modern coordinate-free presentation of parts of these
papers which are interesting for us. We shall follow this presentation in the situation of
a homogeneous quotient manifold $M=G/H$ of a connected Lie group $G$ by
its compact subgroup $H$ with $G$-invariant completely nonholonomic
distribution $D$ and Riemannian metric $(\cdot,\cdot)$ on $M.$ 

Notice that we need the metric $(\cdot,\cdot)$ on all $M$ only in order to define below 
shortly a rigging $D^{\perp}$ of distribution $D.$

Below $\bf{T},$ $\bf{H}$ and $\bf{V}$ denote $C^{\infty}$-modules 
of vector fields on $M,$ tangent respectively to distributions $TM,$ 
$D$ and $D^{\perp}.$ Then $\bf{T}=\bf{H}\oplus\bf{V},$ i.e. any 
vector field $X\in \bf{T}$ is uniquely presented in a view  
$X=HX+VX,$ where $HX\in \bf{H},$ $VX\in \bf{V},$ and $(HX,VX)=0.$ Let
$\overline{\nabla}$ be the Levi-Civita connection of the Riemannian manifold 
$(M,(\cdot,\cdot))$ and  $\langle\cdot,\cdot\rangle=(\cdot,\cdot)_D.$  One can easily check that the formula
\begin{equation}
\label{nabla} 
\nabla_XY:=H\overline{\nabla}_XY,\quad X, Y\in \bf{H}
\end{equation}
defines a metric connection without torsion on $\bf{H}$, i.e. for $X,Y,Z\in \bf{H},$
$$X\langle Y,Z\rangle=\langle\nabla_XY,Z\rangle + \langle Y,\nabla_XZ\rangle,$$
$$T_{\nabla}(X,Y):=\nabla_XY-\nabla_YX-H[X,Y]=0.$$
Moreover $\nabla$ depends on $\langle\cdot,\cdot\rangle$ 
and the rigging $D^{\perp}$ but does not depend on $(\cdot,\cdot)_{D^{\perp}};$ 
$\nabla$ is the unique metric connection without torsion on $\bf{H}$ for
$\langle\cdot,\cdot\rangle$ and $D^{\perp}$ \cite{Gorb}.

{\it The Schouten tensor for nonholonomic manifold} 
$(D,\langle\cdot,\cdot\rangle)$ is an analogue of the curvature tensor for 
Riemannian manifolds and defined as follows
\begin{equation}
\label{scho}
K(X,Y)Z=\nabla_X\nabla_YZ-\nabla_Y\nabla_XZ-H[V[X,Y],Z],\quad X,Y,Z\in \bf{H}.
\end{equation}
Wagner wrote in \cite{Vag}: ``The Schouten tensor does not justify his title
``the curvature tensor'' already because on the ground of his properties 
one cannot judge on the curvature of nonholonomic manifold, i.e. on 
the absence of absolute parallelism'' (for the connection $\nabla$).

Before defining {\it the curvature tensor by Wagner} (or {\it Wagner-Schouten} 
as in \cite{Gorb}) one needs to introduce some mappings and corresponding notations. 

There exists a strongly increasing sequence of $C^{\infty}(M)$-modules 
$${\bf H}_0:={\bf H}, {\bf H}_i\subset {\bf H}_{i+1}:=
{\bf H}_i+[{\bf H}_i,{\bf H}_i],\dots , {\bf H}_r={\bf T}$$ 
of vector fields on the $M$ tangent to the corresponding distributions $D_0=D$, 
$D_i\subset D_{i+1}$, $D_r=TM$ on $M,$ $r$ is {\it the nonholonomy
order} of distribution $D$. Using the scalar product $(\cdot,\cdot)$ on $TM,$ 
we get decompositions
\begin{equation}
\label{dec} 
D_{i+1}=D_i\oplus \Theta_i,\quad \Theta_i:=
D_i^{\perp}\cap D_{i+1},\quad TM=\oplus_{i=0}^{r-1}\Theta_i \oplus D
\end{equation}
and a unique morphism of vector bundles   
$\theta_i: D_{i+1}/D_i \rightarrow \Theta_i\subset D^{i+1},$ 
$i=0,\dots, r-1,$ the right inverse to the canonical morphism 
$\pi_i: D_{i+1}\rightarrow D_{i+1}/D_i.$ There is also a surjective morphism 
$\delta_i: \Lambda^2 D_i\rightarrow D_{i+1}/D_i,$ prescribed by linear
combinations of mappings $X\wedge Y\rightarrow [X,Y]\mod D_i$ for $X,Y\in {\bf H}_i.$

Further, following \cite{Vag} and \cite{Gorb}, it is defined canonically a
new unique invariant Riemannian metric $\{\cdot,\cdot\}$ on $M,$ whose restriction
$\{\cdot,\cdot\}_{|D}=\langle\cdot,\cdot\rangle$ and the last decom\-position in
(\ref{dec}) is orthogonal. For this it is enough to assign $\{\cdot,\cdot\}_{|\Theta_i}$  
by induction on $i=0,\dots,r-1$. A scalar product $g$ on a vector space $V$ defines
non-degenerate linear mapping $g: V\rightarrow V^{\ast}$ such that  
$g(x,y)=g(x)(y)=g(y)(x),$ $x,y\in V,$ and it is defined by it. It is not difficult
to check that $g$ defines the scalar product
\begin{equation}
\label{scp}
g^{\wedge}: \Lambda^2V\rightarrow (\Lambda^2V)^{\ast}|
\quad g^{\wedge}=\phi\circ \Lambda^2 g,
\end{equation}
where $\phi: \Lambda^2V^{\ast}\rightarrow (\Lambda^2V)^{\ast}$ is the canonical isomorphism:
$$\phi(f\wedge h)(x\wedge y)=f(x)h(y)-f(y)h(x),\quad 
x,y \in V,\quad f,h\in V^{\ast}.$$
More explicitly,
$$(u_1\wedge v_1,u_2\wedge v_2)=(u_1,u_2)(v_1,v_2)-(u_1,v_2)(v_1,u_2).$$
By definition,
\begin{equation}
\label{sci}
\{\cdot,\cdot\}^{-1}_{|\Theta_i}:=\theta_i\circ\delta_i\circ
((\{\cdot,\cdot\}_{|D_i})^{\wedge})^{-1}\circ (\theta_i\circ\delta_i)^{\ast}.
\end{equation}
Let us define also a morphism of vector bundles
\begin{equation}
\label{morp}
\mu_i: D_{i+1}\rightarrow \Lambda^2D_i|\quad \mu_i = 
((\{\cdot,\cdot\}_{|D_i})^{\wedge})^{-1}\circ
\delta_i^{\ast}\circ \theta_i^{\ast}\circ \{\cdot,\cdot\}_{|\Theta_i}
\circ \theta_i\circ\pi_i.
\end{equation}
Let us denote by $\stackrel{i}{\nabla}$ a metric connection without torsion on  
$(D_i,\{\cdot,\cdot\}_{|D_i})$ and $H_i,$ $V_i$ are projections playing the same role
for $D_i$, $D_i^{\perp}$, as $H,$ $V$ for $D$, $D^{\perp}$. Now we are ready to
introduce {\it the Wagner covariant derivative}.  

Let us set $\stackrel{0}{K}=K$ (the Schouten tensor) and define 
$\stackrel{1}{\bigcirc}:{\bf H}_1\times {\bf H}_0\rightarrow {\bf H}_0$
by the condition
\begin{equation}
\label{one}
\stackrel{1}{\bigcirc}_XY = \stackrel{0}{\nabla}_{HX}Y+\stackrel{0}{K}
(\mu_0(X))(Y)+[VX,Y], \quad X\in {\bf H}_1,\quad Y\in {\bf H}_0
\end{equation}
and $\stackrel{1}{K}: \Lambda^2{\bf H}_1\rightarrow \End({\bf H}_0)$ 
for $X, Y\in {\bf H}_1$, $Z\in {\bf H}_0$ by condition
\begin{equation}
\label{kone}
\stackrel{1}{K}(X\wedge Y)(Z) = \stackrel{1}{\bigcirc}_
{[X}\stackrel{1}{\bigcirc}_{Y]}Z
- \stackrel{1}{\bigcirc}_{H_1[X,Y]}Z-H[V[X,Y],Z].
\end{equation}
Further, by induction,
$\stackrel{i+1}{\bigcirc}:{\bf H}_{i+1}\times {\bf H}_i\rightarrow {\bf H}_i,$
\begin{equation}
\label{one}
\stackrel{i+1}{\bigcirc}_XY = \stackrel{i}{\nabla}_{H_iX}Y+\stackrel{i}{K}(\mu_i(X))(Y)+[V_iX,Y];
\end{equation}
$\stackrel{i+1}{K}: \Lambda^2{\bf H}_{i+1}\rightarrow \End({\bf H}_i)$;
\begin{equation}
\label{kone}
\stackrel{i+1}{K}(X\wedge Y)(Z) = \stackrel{i+1}{\bigcirc}_
{[X}\stackrel{i+1}{\bigcirc}_{Y]}Z
- \stackrel{i+1}{\bigcirc}_{H_{i+1}[X,Y]}Z-H_i[V_i[X,Y],Z].
\end{equation}
Let us call $\stackrel{1}{\bigcirc},$ $\dots,$ $\stackrel{r-1}{\bigcirc}$ 
intermediate Wagner connections, $\stackrel{r}{\bigcirc}$ {\it the Wagner connection},
and the curvature tensor $\stackrel{r}{K}$ as {\it the Wagner curvature tensor of 
strongly rigged completely nonholonomic distribution} $D.$  
The distribution $D$ possesses absоlute parallelism with respect to $\nabla$
if and only if the Wagner curvature tensor of the distribution $D$ is equal
to zero \cite{Vag}, \cite{Gorb}.

Solov'ev introduced in the paper \cite{Sol84} the notion of a {\it curvature
tensor of distribution on the Riemannian manifold}. In particular, he obtained some special 
properties of the curvature of horizontal distribution of the Riemannian submersion
and left-invariant distributions on Lie groups with left-invariant Riemannian
metric.

He considers a Riemannian manifold $M$ with metric tensor $(\cdot,\cdot)$, 
its Levi-Civita connection $\nabla,$ smooth distribution $D,$ and
distribution $D^{\perp},$ orthogonal to $D$ relative to $(\cdot,\cdot)$; 
${\bf T},$ ${\bf H},$ ${\bf V}$ are corresponding $C^{\infty}(M)$-modules of
smooth vector fields on $M$, tangent to corresponding distributions  
$TM,$ $D,$ $D^{\perp};$ $H$, $V$ are projections from 
${\bf T}= {\bf H}\oplus {\bf V}$ to ${\bf H}$, ${\bf V}.$ 

{\it The induced connection} of distribution $D$ is  
$\overline{\nabla}_XY=H\nabla_XHY+V\nabla_XVY$, and its  
{\it second fundamental form} is the tensor field  
$h=\nabla - \overline{\nabla}$ \cite{Sol79}, \cite{Sol82}; $h^{+}$ and 
$h^{-}$ are symmetric and skew-symmetric parts of the field $h$ respectively.
It is proved in \cite{Sol82} that {\it the distribution $D$ on the Riemannian 
manifold} $(M,(\cdot,\cdot))$ {\it is totally geodesic (respectively involutive) 
if and only if} $h^{+}(HX,HY)=0$ {\it (respectively} 
$h^{-}(HX,HY)=0$) {\it for all} $X,Y\in {\bf T}$. If $T$ is the torsion tensor
for $\overline{\nabla}$ then
\begin{equation}
\label{tw}
T(X,Y)= -V[X,Y] = -2h^{-}(X,Y),\quad X,Y \in D.
\end{equation}

A diffeomorphism $f: M \rightarrow N$ of Riemannian manifolds is called a
$D$-{\it isometry} \cite{Sol79}, \cite{Sol82}, where $D$ is some 
distribution on $M$, if differential $df$ preserves lengths of vectors 
$v\in D$ and $df(D)\perp df(D^{\perp})$. In \cite{Sol79} it is proved 
\begin{theorem}
\label{isominv}
Every $D$-isometry ``preserves'' the expression of view  
$(\overline{\nabla}_XY,Z)-\frac{1}{2}(X,T(Y,Z))$ if 
$X\in {\bf T},$ $Y,Z\in {\bf H}.$
\end{theorem}
Therefore in \cite{Sol84} Solov'ev defines on a Riemannian manifold $(M,(\cdot,\cdot))$ 
with distribution $D$ a new linear connection $С,$ setting
\begin{equation}
\label{ncon}
(C_XHY,Z)=(\overline{\nabla}_XHY,HZ)-(1/2)(X,T(HY,HZ))
\end{equation}
and $C_XVY$ arbitrary for any $X,Y,Z\in {\bf T}.$ We shall suppose that $C_XVY=0.$
Then $C_X({\bf V})=0,$ $C_X({\bf H})\subset {\bf H}$ and by (\ref{tw}), 
$$C_{HX}HY=\overline{\nabla}_{HX}HY=H\nabla_{HX}HY,\quad 
(C_{VX}HY,Z)=-(1/2)(VX,T(HY,HZ)).$$
By definition, {\it the curvature tensor} of the distribution $D$ is $K=\hat{K}\circ H,$ 
where $\hat{K}$ is the curvature tensor of the connection $C$. It is stated in \cite{Sol84} 
without proof that {\it this curvature tensor is the Schouten curvature tensor if $D$ 
is totally geodesic}. Let $\overline{R}$ be the curvature tensor of the connection 
$\overline{\nabla}.$ Then 
\begin{equation}
\label{ktone}
(K(X,Y)Z,W)=(\overline{R}(X,Y)Z,W)-(1/2)(T(X,Y),T(Z,W))
\end{equation}
for any $X,Y,Z,W\in {\bf H}$ and therefore for any such vector fields
\begin{eqnarray}
\label{kttwo}
(K(X,Y)Z,W)=(R(X,Y)Z,W)-2(h^{-}(X,Y),h^{-}(Z,W))+\nonumber\\ 
(h(X,W),h(Y,Z)) - (h(Y,W),h(X,Z)),
\end{eqnarray}
where $R$ is the curvature tensor of the Riemannian manifold $(M,(\cdot,\cdot)).$ The
equation (\ref{kttwo}) defines completely the value $K(HX,HY)HZ$ since $D$ is
parallel with respect to $C$ and therefore $K(X,Y)HZ\in {\bf H}$ for any $X,Y,Z\in {\bf T}.$
It may be considered as an analogue of the Gauss equation for a submanifold.

On the base of formula (\ref{ktone}) or (\ref{kttwo}) are given (completely analogous 
to the Riemannian case) definitions of {\it sectional} $K_{uv}$ and {\it Ricci} $k_w$ 
{\it curvatures} in the direction of two-dimensional subspace $\parallel u\wedge v\neq 0$ 
and one-dimensional subspace $\parallel w\neq 0$ for $u,v, w\in D(p)$ and 
{\it scalar curvature} $s$ at a point $p.$ The sectional curvature of a two-dimensional
distribution is called its {\it Gaussian curvature}. In consequence of the definitions,
these curvatures of the distribution $D$ are invariant relative to any $D$-isometry.

{\it The sectional torsion} for $0\neq u\wedge v \subset D(p)$ is defined in \cite{Sol79} 
by the equality $t_{uv}=\|T(u,v)\|^2/\|u\wedge v\|^2,$ where 
$\|u\wedge v\|^2=\|u\|^2\|v\|^2-(u,v)^2.$ Let $U$ be the domain of exponential mapping
$\overline{\exp}_p$ of the connection $\overline{\nabla}.$ The submanifold 
$E(p)=\overline{\exp}_p(U\cap D(p))$ is called {\it the osculation geodesic surface} 
\cite{Sol79} of the distribution $D$ at the point $p.$ In Theorem 1.3 from \cite{Sol84} 
is established the following geometric interpretation: 
{\it $K_{uv}=K^{(1)}_{uv}+(3/4)t_{uv}$, where $K^{(1)}$ is the sectional curvature
of the surface $E(p)\subset (M,(\cdot,\cdot))$.}

With the help of this interpretation, formula (\ref{tw}), known connection \cite{On} of
sectional curvatures in the total space and the base of Riemannian submersion, and the 
complete geodesic property of horizontal distribution of Riemannian submersion it is
established the following (Theorem 2.4 from \cite{Sol84})
\begin{theorem}
\label{subm}
Let $\pi: (M,(\cdot,\cdot))\rightarrow (B,\{\cdot,\cdot\})$ be a Riemannian submersion, 
$D$ and $D^{\perp}$ respectively its horizontal and vertical distributions on $M.$
Then for any non-collinear vectors $u,v\in D(p),$ $p\in M,$ 
$K_{uv}=K^B_{d\pi(u)d\pi(v)},$ where $K^B$ is the sectional curvature of the Riemannian manifold
$(B,\{\cdot,\cdot\}).$ 
\end{theorem}

\begin{remark}
\label{trsym}
Application of this theorem to sub-Riemannian manifolds includes the case of 
{\it sub-Riemannian manifolds with transverse symmetries} considered in \cite{BaudGar}.
\end{remark}

The following theorem transmits the content of Lemma 4.1 in \cite{Sol84}.
\begin{theorem}
\label{grlie}
Let $G$ be a Lie group with left-invariant Riemannian metric $(\cdot,\cdot)$ and 
distribution $D;$ $e_i,$ $i=1,\dots, m,m+1,\dots, n$ an orthonormal basis of
left-invariant vector fields on $(G,(\cdot,\cdot)),$ $m=\dim D,$ 
$c_{ijk}=([e_i,e_j],e_k).$ Then for $a\neq b,$ $1\leq a,b\leq m,$ 
\begin{eqnarray}
\nonumber K_{e_ae_b}=(1/2)\sum_{i=1}^nc_{abi}(c_{bia}+c_{iab})+
\sum_{j=1}^{m}\{(1/4)(c_{jab}+c_{jba})^2-(3/4)(c_{abj})^2-c_{jaa}c_{jbb}\}.
\end{eqnarray}
\end{theorem}
The following Proposition 4.7 from \cite{Sol84} is valid:
\begin{theorem}
\label{grliebi}
The curvature tensor of any left-invariant distribution $D$ on the Lie group $G$ 
with bi-invariant Riemannian metric is equal to  
\begin{eqnarray}
K(X,Y)Z=(1/4)H[X,H[Y,Z]]+ (1/4)H[Y,H[Z,X]] + \nonumber\\
\nonumber (1/2)H[Z,H[X,Y]] + H[Z,V[X,Y]],\quad X,Y,Z\in {\bf H}.
\end{eqnarray}
\end{theorem} 

Other papers by Solov'ev on the same subject are \cite{Sol85}, 
\cite{Sol86}, \cite{Sol87}.

\section{Contact and symplectic structures}

A smooth differential 1-form $\omega$ on a smooth manifold $M=M^{2k+1}$ 
is called {\it contact} if $\omega\wedge\Lambda^kd\omega\neq 0$ everywhere on
$M.$ A manifold with a contact form is called {\it contact} \cite{Blair}. 
By theorem of G.~Darboux, any point of a contact manifold is contained in some
neighborhood $U$ with coordinates $x^1,\dots,x^k,x^{k+1},y_1,\dots,y_k$ such that 
in these coordinates $\omega|_U=\sum_{i=1}^ky_idx^i+dx^{k+1}$ \cite{Stern}. 

{\it A contact distribution} on a contact manifold $(M,\omega)$ is the null 
set of its contact form, i.e.
\begin{equation}
\label{cont} 
D=\bigcup_{x\in M}D(x),\quad D(x)=\{v\in T_xM| \omega_x(v)=0\}.
\end{equation}
It is clear that $D$ is a smooth vector hyperdistribution on $M$.

\begin{theorem}
\label{rigg}
A contact distribution on any contact manifold is completely nonholonomic and has a canonical
rigging.
\end{theorem} 

\begin{proof}
Since the form $\omega\wedge\Lambda^kd\omega$ is non-degenerate and $M$ is odd-dimensional,
the following statements are valid: 

1) for any point $x\in M$ there exists a unique vector $w_x\in T_xM$ such that 
$\omega_x(w_x)=1$ and $d\omega(w_x,v)=0$ for all $v\in T_xM$;  

2) if $u\in D(x)$ is a non-zero vector, then there exists $v\in D(x)$: 
$d\omega(u,v)\neq 0.$

A vector field $W$ on $(M,\omega)$ such that $W(x)=w_x,$ is called 
{\it a Reeb vector field}. The distribution $D^{\perp}$ on $M,$ spanned by
the Reeb vector field is a canonical rigging of distribution $D.$

Let $U,V\in D$ be arbitrary vector fields on $M$ such that  
$U(x)=u$ and $V(x)=v$ for vectors $u,v\in D(x)$ from p. 2) above.
Then \cite{Stern}
$$d\omega(U,V)=U(\omega(V))-V(\omega(U))-\omega([U,V])=-\omega([U,V])$$
and $[U,V](x)\notin D(x),$ which proves that the distribution $D$ is completely nonholonomic.
\end{proof}

\begin{theorem}
\label{invcd}
Any contact distribution is invariant with respect to the local one-parameter transformation
group generated by the Reeb vector field.
\end{theorem} 

\begin{proof}
This arises from the following inequalities for $X$ tangent to $D$
\begin{equation}
\label{neccont}
0=d\omega(W,X)=W(\omega(X))-X(\omega(W))-\omega([W,X])= \omega(-[W,X])=0
\end{equation}
and the fact that $-[W,X]$ is the {\it Lie derivative of the vector field $X$ in the direction
of the vector field $W$} \cite{Stern}.
\end{proof} 
Notice that a non-zero differential 1-form, proportional to a contact form, 
is itself a contact form. Therefore the Reeb vector field depends on
the contact form.

A smooth closed differential 2-form $\sigma$ on a smooth manifold $M=M^{2k}$ 
is called {\it symplectic}, if its $k$-th exterior degree  
$\Lambda^k\sigma=\sigma\wedge\dots\wedge\sigma\neq 0$ everywhere on $M.$
A smooth manifold with a symplectic form is called {\it symplectic} 
\cite{Blair}. By theorem of Darboux, for any point of any symplectic manifold
in some its neighborhood $U$ there exist coordinates 
$x^1,\dots,x^k,y_1,\dots,y_k$ such that 
$\sigma=\sum_{i=1}^kdy_i\wedge dx^i$ \cite{Stern}. A non-zero differential
2-form, proportional to a symplectic form, is itself symplectic.  

We shall need {\it the canonical symplectic form on the cotangent bundle} 
$T^{\ast}M$ over an arbitrary smooth manifold $M=M^n$ \cite{Bes}. Let 
$p_M: TM\rightarrow M,$ $p^{\ast}_M:T^{\ast}M\rightarrow M,$ and 
$p_{T^{\ast}M}: TT^{\ast}M\rightarrow T^{\ast}M$ be the canonical projections. 
There exists a unique {\it Liouville form} $\alpha$ on $T^{\ast}M$ such that
$$\alpha(\Lambda)=p_{T^{\ast}M}(\Lambda)(dp^{\ast}_M(\Lambda)),\quad 
\Lambda\in TT^{\ast}M.$$
Clearly $\alpha=\sum_{i=1}^n\xi_idx^i$ in natural coordinates  
$x^i,\xi_i;$ $i=1,\dots,n$ on $T^{\ast}M.$ By definition, $\sigma=d\alpha,$ 
i.e. $\sigma=\sum_{i=1}^nd\xi_i\wedge dx^i$ in natural coordinates.

In the case of a homogeneous (sub-)Riemannian manifold $(M^n=G/H,d)$ (and
not only in this case), the search problem for shortest arcs and geodesics
locally reduces to a time-optimal problem which is formulated as follows in
local coordinates $x,\xi$ on $T^{\ast}M$ \cite{PBGM}. We are given a smooth
mapping $f: X\times E^m\rightarrow \mathbb{R}^n,$ $2\leq m < n,$ such that 
$f(x,\cdot)$ is a linear monomorphism for any $x\in X$ and {\it the
Pontryagin-Hamilton function} 
$$H(x,\xi,u)=\sum_{i=1}^n\xi_if^{i}(x,u),\quad (x,\xi,u)\in 
X\times \mathbb{R}^n\times E^m.$$
If $x=x(t)$ is a geodesic parametrized by arc length in $(M,d)$ then there
exists a continuous function $\xi=\xi(t)\neq 0$ such that for almost all
$t$ there exist $u(t)$, the derivatives 
\begin{eqnarray}
\label{difgeod}
\stackrel{.}{\xi_j}(t)=
-\frac{\partial  H((x,\xi,u)(t))}{\partial x_j}= - \sum_{i=1}^{n}\xi_i(t)
\frac{\partial f^i(x(t),u(t))}{\partial x^j},\nonumber\\ 
\stackrel{.}{x}(t)=\frac{\partial H((x,\xi,u)(t))}{\partial \xi}=
f(x(t),u(t)), \quad \|u(t)\|=1, 
\end{eqnarray}
and the following condition is fulfilled   
$$\quad H(x(t),\xi(t),u(t))=\max \{H(x(t),\xi(t),u)| \|u\|\leq 1\}\equiv
M_0\geq 0.$$

A geodesic $x=x(t)$ in $(M,d)$ is called \textit{normal} if $M_0> 0$ and \textit{abnormal} 
if $M_0=0.$ It is called \textit{strictly abnormal} if there is no 
covector function $\xi=\xi(t)$ for which it is normal extremal in $(G,d).$ 
As was shown in the paper \cite{LS} by W.~Liu and H.~Sussman, geodesics of 
a left-invariant sub-Riemannian metric $d$ on a Lie group $G$ could be 
strictly abnormal. We shall consider their example at the end of this paper.
 
In any case the ODE (\ref{difgeod}) defines the {\it Hamiltonian system} 
(vector field)
\begin{equation}
\label{hams}
\stackrel{.}{x}=f(x,u),\quad \stackrel{.}{\xi_j}=- \sum_{i=1}^{n}\xi_i
\frac{\partial f^i(x,u)}{\partial x^j}. 
\end{equation}
Therefore $H(x,\xi,u)=\alpha(\phi(x,\xi,u)),$ where 
$\phi(x,\xi,u)=(\stackrel{.}{x},\stackrel{.}{\xi})(x,\xi,u).$ 

In the case of a Lie group $G$ with Lie algebra $\mathfrak{g}$ we understand
elements of the pair $(\xi,u)\in \mathfrak{g}^{\ast}\times \mathfrak{g}$ respectively
as a left-invariant 1-form and a left-invariant vector field on $G$. Then the
Pontryagin-Hamilton function $H(\xi,u)=\xi(u)$ is defined on 
$\mathfrak{g}^{\ast}\times \mathfrak{g}$. In \cite{Ber14} we proved the following
\begin{theorem}
\label{hlg}
For any Lie group $G$ with Lie algebra $\mathfrak{g},$ the Hamiltonian system for
the function $H$ on $\mathfrak{g}^{\ast}\times \mathfrak{g}$ takes the form
\begin{equation}
\label{gpr}
\stackrel{.}{g}=dl_g(u), \quad g\in G,\quad u\in \mathfrak{g},
\end{equation}
\begin{equation}
\label{xipr}
\stackrel{.}{\xi}(w)=\xi([u,w]),\quad u, w \in \mathfrak{g}.
\end{equation}
\end{theorem}

For a fixed $u\in \mathfrak{g}$, in correspondence with ODEs (\ref{gpr}) and 
(\ref{xipr}), there is defined a vector field $U$ on $T^{\ast}M$: 
$U(g,\xi)=(\stackrel{.}{g},\stackrel{.}{\xi}(\cdot))=(dl_g(u),\xi([u,\cdot])).$
Let $V$ be an analogous vector field on $T^{\ast}M$, defined by an element 
$v\in \mathfrak{g}.$ Then
$$\sigma(U,V)((g,\xi))=d\alpha(U,V)((g,\xi))=
(U(\alpha(V))-V(\alpha(U))-\alpha([U,V]))((g,\xi))=$$
$$\xi([u,v])-\xi([v,u])-\xi([u,v])=\xi([u,v]).$$
Thus,
\begin{equation}
\label{hambr}
\sigma(U,V)((g,\xi))=\xi([u,v]).
\end{equation}

Below, the differential of any smooth mapping $f$ of smooth manifolds is denoted by $df$. 
Let us define the following mappings for the Lie group $G$ \cite{BN}:
$$I(g): G\rightarrow G;\quad I(g)(g')=gg'g^{-1},$$
$$\Ad g=dI(g)_e: \mathfrak{g}=T_eG\rightarrow T_eG= \mathfrak{g},$$
$$\ad=(d\Ad)_e,\quad \ad u(v)=[u,v],\quad u,v\in \mathfrak{g},$$
$$\Ad^{\ast} g:\mathfrak{g}^{\ast}\rightarrow \mathfrak{g}^{\ast},\quad 
\Ad^{\ast}g\xi(v)=\xi((\Ad g)^{-1}(v)),\quad v\in \mathfrak(g),\quad \xi\in 
\mathfrak{g}^{\ast},$$
$$\ad^{\ast}u \xi(v)=\xi(-\ad u(v))=\xi([v,u]),\quad 
u,v\in \mathfrak{g},\quad \xi\in \mathfrak{g}^{\ast}.$$  

\begin{theorem}
\label{mai}
Let $G$ be a Lie group with Lie algebra $\mathfrak{g}$ and unit $e$, 
$\xi\in \mathfrak{g}^{\ast}=T^{\ast}_{e}G$ co-vector, $\Ad^{\ast}\xi(g):= 
\Ad^{\ast}g(\xi),$ $g\in G,$ the action of co-adjoint representation 
of the Lie group $G$ to the co-vector $\xi.$ Then
\begin{equation}
\label{diff}
(d(\Ad^{\ast}\xi)(w))(v)=\ad^{\ast}u(\Ad^{\ast}g_0\xi(v)),
\end{equation}
if
\begin{equation}
\label{data} 
u,v\in \mathfrak{g},\quad w=dl_{g_0}(u)\in T_{g_0}G,\quad g_0\in G.
\end{equation}
\end{theorem}

\begin{proof} 
By Ado theorem on existence of exact matrix representation of any Lie algebra,
the third theorem of Lie is valid (Theorem 2.9 in \cite{BN}). Then every 
Lie group is locally isomorphic to a matrix Lie group, possibly, with a strengthened 
topology (see details in Theorem 1 from \cite{Ber14}). Therefore we can suppose that
$G$ is a matrix Lie group. Then 
$\Ad g(v)=gvg^{-1},$ $dl_{g}(u)=gu$ if $u,v\in \mathfrak{g}$ и $g\in G$.

\begin{lemma}
\label{lem}
Let $g=g(t),$ $t\in (a,b)$ be a smooth path in the Lie group $G.$ Then
\begin{equation}
\label{eq}
(g(t)^{-1})'=-g(t)^{-1}\cdot g'(t)\cdot g(t)^{-1}.
\end{equation}
\end{lemma}

\begin{proof}
Differentiating $g(t)\cdot g(t)^{-1}=e$ by $t,$ we get
$$0=(g(t)\cdot g(t)^{-1})'=g'(t)\cdot g(t)^{-1}+g(t)\cdot (g(t)^{-1})',$$
whence immediately follows (\ref{eq}).
\end{proof}

To prove Theorem \ref{mai}, we choose a smooth path $g=g(t),$
$t\in (-\varepsilon,\varepsilon),$ in the Lie group $G$ such that $g(0)=g_0,$ 
$g'(0)=w.$ Then by Lemma \ref{lem},
$$(d(\Ad^{\ast}\xi)(w))(v)=(\xi(g(t)^{-1}\cdot v\cdot g(t)))'(0)=$$
$$\xi((g(t)^{-1}\cdot v\cdot g(t))'(0))=\xi((g(t)^{-1})'(0)\cdot v\cdot g_0+g_0^{-1}\cdot v\cdot g'(0))=$$
$$\xi(-(g_0^{-1}g'(0)g_0^{-1})\cdot v\cdot g_0 +g_0^{-1}\cdot v\cdot g'(0))=$$
$$\xi(-u\cdot(g_0^{-1}\cdot v\cdot g_0) + (g_0^{-1}\cdot v\cdot g_0)\cdot u)=$$
$$\xi(-\ad u((\Ad g_0)^{-1}(v))=\ad^{\ast}u\xi((\Ad g_0)^{-1}(v))=
\ad^{\ast}u(\Ad^{\ast}g_0\xi(v)).$$
\end{proof}

Theorems \ref{hlg} and \ref{mai} immediately imply
\begin{corollary}
\label{adj}
For any connected Lie group $G$ and $\xi_0\in T_e^{\ast}G$ the mapping
$$(\Ad^{\ast}(\cdot))^{-1}(\xi_0): g\in G\rightarrow 
(\Ad^{\ast}g)^{-1}\xi_0\in T^{\ast}_gG$$ 
is a unique section of the bundle $p^{\ast}_G: T^{\ast}G\rightarrow G$, which is a solution of
the Hamiltonian system (\ref{gpr}), (\ref{xipr}) with an initial value $\xi_0$ at $e\in G$.   
\end{corollary}

\begin{definition}
\label{coad}
The mapping $(\Ad^{\ast}(\cdot))^{-1}: G\rightarrow (\mathfrak{g}^{\ast}\rightarrow \mathfrak{g}^{\ast})$ is called the co-adjoint representation of the Lie group $G;$ 
$\ad^{\ast}(\cdot): \mathfrak{g}\rightarrow (\mathfrak{g}^{\ast}\rightarrow \mathfrak{g}^{\ast})$ is the co-adjoint representation of the Lie algebra $\mathfrak{g}.$ The image of the mapping 
$(\Ad^{\ast}(\cdot))^{-1}(\xi_0)$ is called the orbit of element $\xi_0$ relative to
the co-adjoint representation of the Lie group $G$.
\end{definition}

Every non-trivial orbit of the co-adjoint representation of the Lie group admits a 
{\it canonical symplectic structure} \cite{Kir72}. Let $O(\xi)$ be the orbit of
an element $\xi\in \mathfrak{g}^{\ast}$ relative to the co-adjoint representation of the
Lie group $G$ with Lie algebra $\mathfrak{g}.$ Then 
$T_{\xi}O(\xi)=\{\eta=\ad^{\ast}u\xi| u\in \mathfrak{g}\}.$
By definition,
\begin{equation}
\label{coads}
\sigma(\eta_1,\eta_2)=\xi([u_1,u_2]),\quad\mbox{if}\quad\eta_1=
\ad^{\ast}u_1\xi, \eta_2=\ad^{\ast}u_2\xi.
\end{equation}
It is not difficult to check that this definition does not depend on the presentation of
the elements $\eta_1,$ $\eta_2.$ Obviously, $\sigma$ is skew-symmetric. One can easily check
also that it is non-degenerate. It follows from the Jacobi identity in the Lie algebra 
$(\mathfrak{g},[\cdot,\cdot])$ that the differential form $\sigma$ is closed \cite{Kir72}.

Theorem \ref{mai} implies the invariance of the form (\ref{coads}) relative to  
$\Ad^{\ast}G.$ Theorems \ref{hlg}, \ref{mai} and Corollary \ref{adj} show that
the coincidence of right parts in formulae (\ref{hambr}) and (\ref{coads}) is
not occasional. In particular, one can define the canonical symplectic form on 
orbits of the co-adjoint representation otherwise with the help of the symplectic form 
$\sigma$ on $T^{\ast}G;$ and the exactness of the second form implies that the first
one is closed. 

\begin{remark}
\label{sof}
Theorems 2 and 3 from Lecture 7.3 in \cite{Kir02} give yet another alternative 
approach to construct the canonical symplectic structure on co-adjoint orbits, 
based on the notion of Poisson manifold. A.~Weinstein remarked that Theorem 2
was formulated by Lie approximately in 1890. A.A.~Kirillov supposed that
Lie in no way used this result. F.A.~Berezin reopened this theorem in 1968 when
he investigated universal enveloping algebras (see quotations in \cite{Kir02}). 
\end{remark}

\section{Riggings of left-invariant distributions on Lie groups}
\label{riggings}

Each left-invariant sub-Riemannian metric on the Lie group $G$ is defined by 
a left-invariant completely nonholonomic vector distribution $D$ and left-invariant
scalar product $(\cdot,\cdot)$ on $D$. The Solov'ev method, for a given rigging 
$D^{\perp}$ of distribution $D,$ gives the unique curvatures of metrized distribution 
$(D,(\cdot,\cdot)).$

Let $D$ be a left-invariant completely nonholonomic distribution of dimension
$m\geq 2$ and codimension $k\geq 1$ on the Lie group $G^n$. Then there are 
left-invariant differential 1-forms $\omega_{m+1},\dots, \omega_n$ such that 
$D=\{v\in TG|\omega_j(v)=0,j=m+1,\dots, n\}.$ It is clear that the forms $\omega_j$ 
are not defined uniquely by distribution $D,$ but they constitute a basis 
over $\mathbb{R}$ of the unique vector space $\Nul(D)$ of 
all left-invariant 1-forms on $G$, which annihilate the distribution $D$.
Let us fix such forms $\omega_{m+1},\dots, \omega_n$ and some basis $X_1,\dots,X_m$
of left-invariant vector fields on $G$ tangent to $D$. 

Similarly, we can define any rigging $D^{\perp}$ of the distribution $D$ if
we choose some left-invariant differential 1-forms $\omega_{1},\dots, \omega_m$ 
on $G,$ which are linearly independent with $\omega_{m+1},\dots, \omega_n,$ and set  
$D^{\perp}=\{w\in TG|\omega_i(w)=0,i=1,\dots, m\}.$ 

The considerations above, especially Relations (\ref{neccont}) and (\ref{hambr}), prompt 
three possible cases of naturally assigning left-invariant rigging $D^{\perp}$ of 
left-invariant distribution $D$ on the Lie group $G$:

There exist 1-forms $\omega_{1},\dots, \omega_m$ on $G,$ linearly independent 
with $\omega_{m+1},\dots,  \omega_n,$ with $\mathbb{R}$-linear span depending
only on $D,$ satisfying one of the following three conditions: 

1) If $\omega_1(W)=\dots = \omega_m(W)=0$ for a left-invariant vector field $W$ on 
$G,$ then $\omega_i[W,X_j]=0$ for all $j=1,\dots, m$ and $i=m+1,\dots, n.$

2) If $\omega_1(W)=\dots = \omega_m(W)=0$ for a left-invariant vector field $W$ 
on $G,$ then $\omega_i[W,X_j]=0$ for all $i,j=1,\dots, m.$

3) $\omega_i[W,X]=0$ for all left-invariant vector fields $W$ and $X$ 
on $G$ and $i=1,\dots, m.$

1) The Jacobi identity implies that the set of left-invariant vector fields 
$X$ on $G$ such that $[X,D]\subset D,$ is a Lie algebra. Therefore
the corresponding $D^{\perp}(e)$ is a Lie subalgebra in $\mathfrak{g}.$ 

2) Since  $D$ is a completely nonholonomic distribution, the Jacobi identity 
implies that the corresponding $D^{\perp}(e)$ is an ideal in $\mathfrak{g}.$ 

Clearly, 3) is a partial case of 2). {\it There are analogues of conditions
1), 2), 3) for homogeneous manifolds $G/H$}. It is necessary to note 
that conditions 1)--3) are very general in two senses: they do not take into 
account a particular structure of homogeneous manifolds $G/H$ with invariant
completely nonholonomic distribution $D$ as well as a sub-Riemannian metric 
connected with them.  Maybe it would be possible to find other 
natural conditions to choose a rigging of $D$ in partial cases of $G/H$ and 
invariant sub-Riemannian metrics on $G/H$. For example, one could use 
some Killing vector fields on homogeneous sub-Riemannian manifolds.

\section{Examples}

By Theorem \ref{rigg}, any contact distribution has a canonical rigging, so
in this case we can apply the Solov'ev definition of curvatures. 

We shall show that any completely nonholonomic left-invariant rank two 
distribution on any three-dimensional Lie group $G$ is contact, so we 
can apply Theorem \ref{grlie} to calculate the sectional curvature for
any left-invariant sub-Riemannian metric on $G$. This curvature 
coincides with Ricci and Gaussian curvatures.

\begin{proposition}
\label{contt}
A three-dimensional Lie group $G$ admits a left-invariant contact form 
$\omega$ with a contact distribution $D$: $\omega(D)=0$ if and only if there 
exists a completely nonholonomic left-invariant rank two distribution 
$D$ on $G$ satisfying condition 1) from Section \ref{riggings}; moreover,
there exists a non-zero left-invariant 1-form $\omega$ on $G$ such that $\omega(D)=0$.
\end{proposition}

\begin{proof}
Necessity follows from Theorems \ref{rigg} and \ref{invcd}.

Sufficiency. Suppose that a left-invariant completely nonholonomic rank two 
distribution $D$ on $G$ together with a unique left-invariant distribution
$D^{\perp}$ satisfy condition 1) from Section \ref{riggings}, i.e. 
$[D^{\perp}(e),D(e)]\subset D(e).$ Then there exists a non-zero left-invariant 
1-form $\omega$ on $G$ which is unique up to multiplication by a constant such
that $\omega(D)=0$ and a unique left-invariant vector field on
$G$ tangent to $D^{\perp}$ such that $\omega(W)=1.$ Hence for any linearly
independent left-invariant vector fields $X,$ $Y$ on $G$, tangent to $D$,
similarly to the proofs of Theorems \ref{rigg} and \ref{invcd}, we get  
$$d\omega(X,Y)=-\omega([X,Y])\neq 0,$$
$$d\omega(W,X)=-\omega(W,X)=0,\quad d\omega(W,Y)=-\omega(W,Y)=0.$$
This means that $\omega$ is a contact form on $G$ and $W$ is the
Reeb vector field for $\omega$.
\end{proof}

\begin{proposition}
\label{everyct}
1) There is no left-invariant completely nonholonomic rank two distribution on
a three-dimensional Lie group $G$ if and only if $G$ is commutative or its
Lie algebra $\mathfrak{g}$ admits a basis $e_1,e_2,e_3$: 
$[e_1,e_2]=e_2$, $[e_1,e_3]=e_3$, $[e_2,e_3]=0$. 

2) There are four types of mutually non-isomorphic connected commutative Lie 
groups, and they are unimodular. There exists only one connected Lie group with the 
Lie algebra of second form; it is simply connected, solvable, non-unimodular and 
characterized by the property that, supplied by an arbitrary left-invariant 
Riemannian metric, it is isometric to the Lobachevsky space. 

3) Every left-invariant completely nonholonomic rank two distribution on any 
three-dimensional Lie group is contact. 
\end{proposition}

\begin{proof}
1) The sufficiency in the first statement is clear. The necessity follows 
from formula (4.2), table on p. 307, and Lemma 4.10 in the paper \cite{Miln} 
by Milnor. There are given the Lie brackets for special {\it Milnor bases} 
$e_1,e_2,e_3$ in the Lie algebras and a full classification respectively of 
unimodular and non-unimodular Lie algebras. There are six types of unimodular 
Lie algebras and a continuous connected one-parameter family of non-unimodular 
Lie algebras.

2) The statement about commutative groups is trivial; concerning another
statement, see \cite{Miln}.

3) The same formula (4.2), table on p. 307, and Lemma 4.10 in \cite{Miln},
together with Proposition \ref{contt}, imply that 
for any other three-dimensional Lie group $G$, the left-invariant distribution 
$D$ on $G$ with basis $e_1, e_2$ for $D(e)$ is completely nonholonomic and contact 
with respect to a left-invariant contact 1-form $\omega$ on $G$ with the 
left-invariant Reeb vector field $W$ such that $W(e)=e_3.$  
In Lemma 4.10, one needs to take $\alpha=2,$ $\delta=0,$ $\beta \neq 0.$

The proof is completed by a remark from the paper \cite{AgrBar}
by A.~Agrachev and D.~Barilari. It states that in each of the cases under 
consideration but one, all left-invariant bracket generating distributions 
are equivalent by an automorphism of the Lie algebra. The excluded cases 
are Lie groups $G$ with Lie algebra $\mathfrak{sl}(2).$ Besides the 
one considered above, the so-called {\it elliptic} distribution $D$ for $G,$ 
there is a non-equivalent to it, the so-called {\it hyperbolic} distribution 
$D_h$ for $G$ such that the restriction of the Killing form onto $D_h(e)$ 
is sign-indefinite. We can take for $D_h(e)$ the basis $e_2, e_3$. Then
formula (4.2), table on p. 307 in \cite{Miln}, and 
Proposition \ref{contt} imply that $D_h$ is hyperbolic, bracket generating,
and contact with respect to a left-invariant contact 1-form $\omega$ on 
$G$ with left-invariant Reeb vector field $W$ such that $W(e)=e_1.$   
\end{proof}

The Reeb vector field $W$ could generate a local one-parameter subgroup of
isometries for $(G,d)$ if and only if $G$ is locally isomorphic to the Heisenberg 
group $\mathbb{H}^1,$ $SO(3)$ or $SL(2).$ In the last case the corresponding 
distribution $D$ must be elliptic. Also $W$ will be tangent to a closed 
one-dimensional subgroup $H\subset G$ acting on the right by isometries 
in $(G,d).$ Then $G/H$ admits an invariant Riemannian metric $\delta$ such 
that the canonical projection $p:(G,d)\rightarrow (G/H,\delta)$ is a submetry. 
Therefore by Theorem \ref{subm} the Gaussian curvature of $(G,d)$ is equal 
to the constant Gaussian curvature of $(G/H,\delta)$. These groups
with such metric $d$ were studied in \cite{BZ01} --- \cite{BZ17}. There the 
corresponding Gaussian curvatures were equal respectively to $0$, $1$, $-1$ 
what agrees with statements in \cite{BaudGar}. 

Notice that any two sub-Riemannian metrics on $\mathbb{H}^1$ give isometric spaces.
The corresponding distribution $D$ also satisfies both conditions 2) and 
3) from section \ref{riggings}.  

\begin{proposition}
\label{three}
Assume that a left-invariant sub-Riemannian metric $d$ on a Lie group $G$ is defined by 
the scalar product $\langle\cdot,\cdot\rangle$ on distribution $D$ with the 
rigging $D^{\perp}$ satisfying condition 3). Then all curvatures of $(G,d)$ are
equal to zero. 
\end{proposition}

\begin{proof}
Let $(\cdot,\cdot)$ be a left-invariant Riemannian metric on $G$ such that 
 $(\cdot,\cdot)|_D=\langle\cdot,\cdot\rangle,$ 
$(D,D^{\perp})=0,$ $\{e_1,\dots, e_m\},$ $\{e_{m+1},\dots, e_n\}$ an orthonormal bases in
$D$ and $D^{\perp}.$  Then $c_{ijk}=0$ for all $i,j=1,\dots, n$, $k=1,\dots, m$ in the
notation of Theorem \ref{grlie}, which implies Proposition \ref{three}.
\end{proof}

The group $\mathbb{H}^1$ is a partial and the simplest case of the so-called {\it Carnot groups}.

\begin{definition}
\label{carn}
The Carnot group is a Lie group $G,$ supplied by a 
1-parameter multiplicative group of automorphisms $(\delta_s,\cdot)$, $s>0,$ such
that the vector subspace $V:=\{v\in \mathfrak{g}: d\delta_s(v)=sv\}$ generates 
$\mathfrak{g},$ i.e. the least Lie subalgebra in $\mathfrak{g},$ containing
$V,$ coincides with $\mathfrak{g}.$ The expression ``the Carnot group with
a left-invariant sub-Riemannian metric'' means that $D(e)=V.$ 
\end{definition}

\begin{corollary}
\label{carnot}
Any Carnot group $G$ with a left-invariant (sub-)Riemannian metric $d$ defined by
left-invariant distribution $D$ and scalar product $\langle\cdot,\cdot\rangle$ on $D$
(with mentioned rigging $D^{\perp}$ of distribution $D$ if
$G$ is non-commutative) has zero curvatures.  
\end{corollary}

\begin{proof}
Obviously, the statement is true if $G$ is a commutative Lie group because then $d$ is 
a Riemannian metric and $(G,d)$ is locally isometric to an Euclidean space. Otherwise,
the Lie algebra of the Lie group $G$ is a graded nilpotent Lie algebra 
$\mathfrak{g}=\oplus_{k=1}^l \mathfrak{g}_k$ generated by $\mathfrak{g}_1,$ 
where $l\geq 2,$ $D(e)=\mathfrak{g}_1.$  
It is clear that $G$ satisfies condition 3) from Section \ref{riggings} for 
$D^{\perp}(e)=\oplus_{k=2}^l \mathfrak{g}_k$, so we can apply Proposition \ref{three}. 
\end{proof}

\begin{remark}
Corollary \ref{carnot} is obvious for any Carnot group $G$ with a left-invariant 
(sub-)Riemannian metric $d$ because the members of one-parameter 
multiplicative group $(\delta_s,\cdot)$, $s>0,$ from Definition \ref{carn} are $s$-similarities 
of $(G,d).$ Calculations in the above proof of Proposition \ref{three} demonstrate the 
correctness of adopted method for Carnot groups. The statement of Corollary \ref{carnot} is given in 
\cite{BaudGar} only for $G$ of step two.    
\end{remark}

There are the following Hopf bundles  
$$S^{2n+1}=U(n+1)/U(n)\rightarrow U(n+1)/(U(n)\times U(1))=\mathbb{C}P^n, \quad n\geq 1,$$
$$S^{4n+3}= Sp(n+1)/Sp(n)\rightarrow Sp(n+1)/(Sp(n)\times Sp(1))=\mathbb{H}P^n, \quad n\geq 1,$$
$$S^{4n+3}= Sp(n+1)/Sp(n)\rightarrow Sp(n+1)/(Sp(n)\times U(1))=\mathbb{C}P^{2n+1}, \quad n\geq 1,$$
$$S^{15}=Spin(9)/Spin(7)\rightarrow Spin(9)/Spin(8)=\mathbb{C}aP^1=S^8,$$
$$\mathbb{C}P^{2n+1}=Sp(n+1)/(Sp(n)\times U(1))\rightarrow Sp(n+1)/(Sp(n)\times Sp(1))=\mathbb{H}P^n, \quad n\geq 1.$$
The fibres of these bundles are spheres of respective dimensions 
$1,$ $3,$ $1,$ $7,$ and $2.$ 

If we supply all the total spaces (spheres) of the bundles of the first four types 
by the canonical Riemannian metrics of sectional curvature 1, then there are 
unique canonical Riemannian symmetric metrics on the bases of these bundles 
such that the corresponding projections are Riemannian submersions. After that 
there are unique canonical symmetric Riemannian metrics on the bases of the 
bundles of the last type such that the corresponding projections are Riemannian 
submersions. 

Many details on these Riemannian submersions can be found in papers
\cite{Zil82}, \cite{BerNik14}. The next to the last case is the most difficult,
but at the same time the most interesting case, which involves essentially
the Clifford algebras $Cl^n$ and the Cayley algebra
$\mathbb{C}a$ of octonions. The image of the Lie algebra 
$\mathfrak{spin}(7)=\mathfrak{so}(7)$ of the Lie subgroup $Spin(7)$ is not the standard 
inclusion into $\mathfrak{spin}(9)=\mathfrak{so}(9)$, but its image 
$\tau(\mathfrak{so}(7))$ under an outer automorphism $\tau$ of Lie algebra
$\mathfrak{so}(8),$ with standard inclusion $\mathfrak{so}(8)\subset \mathfrak{so}(9),$ 
the so-called {\it triality automorphism} of order 3. In reality 
$\tau$ is induced by a rotation symmetry $s\in S_3$ of the Dynkin diagram
$D_4$ (which is a tripod) of the Lie algebra $\mathfrak{so}(8).$   

Then the horizontal
distributions $D$ of all these Riemannian submersions are completely
nonholonomic in the total spaces of these bundles. We shall get homogeneous
sub-Riemannian metrics on the total spaces with distributions $D$ if we supply
$D$ by the induced scalar products from the previous Riemannian metrics. After this
procedure, not changing the previous symmetric Riemannian metrics on the bases 
of the bundles, we get submetries from the sub-Riemannian manifolds onto the Riemannian
symmetric spaces. In all cases analogues of condition 1) from 
Section \ref{riggings} for horizontal and vertical distributions are satisfied. Therefore,
by Theorem \ref{subm}, we can calculate all curvatures of the total homogeneous 
sub-Riemannian manifolds, using the curvatures of the bases with symmetric 
Riemathennian metrics. In the first three cases there are respective groups 
$U(1),$ $Sp(1),$ and $U(1)$ of {\it transverse symmetries} studied in \cite{BaudGar}. 
In the other cases this is impossible because the spheres $S^7$ and $S^2$  
admit no structure of a Lie group.

Now we shall consider the Liu-Sussman example from Section 9.5 in \cite{LS}. 
Let $G$ be any four-dimensional Lie group whose Lie algebra $\mathfrak{g}$ has two 
generators $f$ and $g$ such that (1) $f,$ $g,$ $[f,g]$ and $[f,[f,g]]$ form a basis 
in $\mathfrak{g};$ (2) $[g,[f,g]]$ belongs to the linear span of vectors $f,$ $g,$ 
and $[f,g];$ (3) $[g,[f,g]]$ does not belong to the linear span of vectors $f$ and 
$[f,g].$ One can take $G=SO(3)\times \mathbb{R}$ with Lie algebra 
$\mathfrak{so}(3)\oplus \mathbb{R}$ and  
\begin{equation}
\label{set}
f=k_1\oplus 1,\quad g=(k_1+k_2)\oplus 2,
\end{equation} 
where $k_1,$ $k_2,$ $k_3$ are generators of the Lie algebra $\mathfrak{so}(3)$ of 
the Lie group $SO(3)$ such that $[k_1,k_2]=k_3,$ $[k_2,k_3]=k_1,$ and $[k_3,k_1]=k_2$. 
One can easily check that 
\begin{equation}
\label{libr}
[f,g]=k_3\oplus 0,\quad [f,[f,g]]=-k_2\oplus 0,\quad [g,[f,g]]=(k_1-k_2)\oplus 0 = 2f-g.
\end{equation}
Therefore all conditions (1),(2),(3) are satisfied. 
The left-invariant sub-Riemannian metric $d$ on $G$ is defined by orthonormal basis 
$\{f,g\}$  on the vector subspace $D(e)\subset \mathfrak{g}.$ By Theorems 5 and 6 
in \cite{LS}, the subgroups $g_1(t)=\exp(tg),$ $g_2(t)=\exp(-tg)$ and their left 
shifts are only strictly abnormal geodesics in $(G,d).$ 

One can easily see from Relations (\ref{set}) and (\ref{libr}) that 
{\it the distribution $D$ does not satisfy any condition 1), 2), or 3)} from 
Section \ref{riggings}.

A simplest case is when $D^{\perp}(e)$ has orthonormal basis $k_3\oplus 0,$ $0\oplus 1$. 
Then by (\ref{set}), (\ref{libr}) and the notation of Theorem \ref{grlie} the only non-zero 
constants are $c_{231}=-c_{321}=2,$
$$c_{123}=-c_{213}=c_{131}=-c_{311}=-c_{132}=c_{312}=-c_{134}=c_{314}=-c_{232}=c_{322}=1.$$
By Theorem \ref{grlie}, all curvatures of $(G,d)$ with this $D^{\perp}(e)$ are equal to $K_{fg}=3/2.$

If we change $0\oplus 1$ by $[f,[f,g]]=-k_2\oplus 0,$ then the only non-zero 
constants are
$$c_{231}=-c_{321}=c_{341}=-c_{431}=2,$$
$$c_{123}=-c_{213}=c_{134}=-c_{314}=-c_{232}=c_{232}=-c_{143}=$$
$$c_{413}=-c_{243}=c_{423}=-c_{342}=c_{432}=-c_{344}=c_{434}=1.$$
By Theorem \ref{grlie}, all curvatures of $(G,d)$ with this $D^{\perp}(e)$ are equal to 
$K_{fg}=1.$

\end{document}